\documentclass{amsart}
\usepackage{amsmath,amsfonts,amssymb,amsthm,epsfig,color, tikz}
\usepackage{hyperref} 
\usepackage{bbm}
\usepackage[margin=1.2in]{geometry}
\newcommand{\bb}{\mathbb}

\newcommand{\Z}{\bb Z}
\newcommand{\R}{\bb R}
\newcommand{\RR}{\bb R}

\newcommand{\xx}{\mathbf{x}}
\newcommand{\pp}{\mathbf{p}}

\newcommand{\vv}{\mathbf{v}}

\newcommand{\s}{\bb S}
\newcommand{\D}{\mathcal D}

\newcommand{\Pa}{\mathcal P}

\newcommand{\wrt}[1]{\mathrm{d}{#1}}
\newcommand{\supp}{\operatorname{supp}}
\newcommand{\spn}{\operatorname{span}}

\newcommand{\Euc}{\operatorname{Euc}}
\newcommand{\vol}{\operatorname{vol}}
\newcommand{\SL}{\operatorname{SL}}

\newcommand{\SO}{\operatorname{SO}}
\newcommand{\Id}{\operatorname{Id}}

\newcommand{\La}{\Lambda}
\newtheorem{Theorem}{Theorem}
\numberwithin{Theorem}{section}
\newtheorem{theo}[Theorem]{Theorem}

\newtheorem{coro}[Theorem]{Corollary}

\newtheorem{lemm}[Theorem]{Lemma}

\newtheorem*{lemma*}{Lemma}
\newtheorem*{question*}{Question}

\newtheorem*{theorem*}{Theorem}

\numberwithin{equation}{section}
\theoremstyle{remark}

\newtheorem{rema}[Theorem]{\sc Remark}
\begin{document}
\title[Spherical Averages and Spiraling]{Spiraling of approximations and spherical averages of Siegel transforms}
\author{Jayadev~S.~Athreya}
\author{Anish Ghosh}
\author{Jimmy Tseng}

\address{J.S.A.: Department of Mathematics, University of Illinois Urbana-Champaign, 1409 W. Green Street, Urbana, IL 61801, USA}
\email{jathreya@illinois.edu}
\address{A.G.: School of Mathematics, Tata Institute of Fundamental Research, Homi Bhabha Road, Mumbai 400005 India}
\email{ghosh@math.tifr.res.in}
\address{J.T.:  School of Mathematics, University of Bristol, University Walk, Bristol, BS8 1TW UK}
\email{j.tseng@bristol.ac.uk}

    \thanks{J.S.A.\ partially supported by NSF grant DMS 1069153, and NSF grants DMS 1107452, 1107263, 1107367 ``RNMS: GEometric structures And Representation varieties" (the GEAR Network).}
    \thanks{A.G. partially supported by the Royal Society}
    \thanks{J.T. acknowledges the research leading to these results has received funding from the European Research Council under the European Union's Seventh Framework Programme (FP/2007-2013) / ERC Grant Agreement n. 291147.}
    
    
\begin{abstract}  We consider the question of how approximations satisfying Dirichlet's theorem spiral around vectors in $\R^d$. We give pointwise almost everywhere results (using only the Birkhoff ergodic theorem on the space of lattices). In addition, we show that for \emph{every} unimodular lattice, on average, the directions of approximates spiral in a uniformly distributed fashion on the $d-1$ dimensional unit sphere. For this second result, we adapt a very recent proof of Marklof and Str\"ombergsson~\cite{MS3} to show a spherical average result for Siegel transforms on $\SL_{d+1}(\bb R)/\SL_{d+1}(\bb Z)$.  Our techniques are elementary.  Results like this date back to the work of Eskin-Margulis-Mozes~\cite{EMM} and Kleinbock-Margulis~\cite{KM} and have wide-ranging applications. We also explicitly construct examples in which the directions are not uniformly distributed. 
\end{abstract}

\maketitle

%

\section{Introduction}\label{sec:intro}\noindent It is a corollary of a classical theorem of Dirichlet \cite{Dirichlet}, that, for every ${\bf x} \in \bb R^d$ ($d \geq 1$), there exist infinitely many $({\bf p}, q) \in \bb Z^d \times \bb N$ such that
\begin{equation}\label{dirichlet}
\|q{\bf x} - {\bf p}\| < C_d|q|^{-1/d}.
\end{equation}
Here,  $\|~\|$ denotes the Euclidean norm on $\bb R^d$ and $C_d$ is a constant depending only on $d$. If the $L^{\infty}$-norm is used in (\ref{dirichlet}), then $C_d$ can be taken to be $1$ for all $d$. In this paper, we are interested in the distribution of the \emph{directions} of the approximates $({\bf p}, q) \in \bb Z^d \times \bb N$ approaching $\bf x$, that is, the quantities $$\theta(\pp, q) := \frac{q{\bf x} - {\bf p}}{\|q{\bf x} - {\bf p}\|} \in \bb S^{d-1}.$$ Given $A \subset \s^{d-1}$, $T >0$, we form the counting functions
$$N(\xx, T) = \#\{({\bf p}, q) \in \bb Z^d \times \bb N, 0 < q \le T: \|q{\bf x} - {\bf p}\| <  C_d|q|^{-1/d}\}$$ and $$N(\xx, T, A) = \#\{({\bf p}, q) \in \bb Z^d \times \bb N, 0 < q \le T: \|q{\bf x} - {\bf p}\| < C_d|q|^{-1/d}, \theta(\pp, q) \in A \}.$$ Note that, while Dirichlet's theorem guarantees that $N(\xx, T) \rightarrow \infty$ as $T \rightarrow \infty$, $N(\xx, T, A)$ could, a priori, be $0$ for all $T >0$. Our first main theorem is

\begin{Theorem}\label{theorem:approx} For $A \subset \mathbb S^{d-1}$, a measurable subset, and for almost every $\xx \in \R^d$, $$\lim_{T \rightarrow \infty} \frac{N(\xx,T, A)}{N(\xx, T)} = \vol(A).$$ Here $\vol := \vol_{\bb S^{d-1}}$ is the Lebesgue probability measure on $\bb S^{d-1}$.

\end{Theorem}

\begin{coro}\label{cor:approx}If $\vol(A) >0$, the inequality  $$\|q{\bf x} - {\bf p}\| < C_d|q|^{-1/d}, \theta(\pp, q) \in A$$ has infinitely many solutions $({\bf p}, q) \in \bb Z^d \times \bb N$ for almost every $\xx \in \R^d$. \end{coro}

\noindent\textbf{Remark.} Logarithmic (in $T$) almost sure (in $\xx$) asymptotics for $N(\xx, T)$ follow from work of W.~Schmidt~\cite{Schmidt}; see also~\cite{APT} for a simple proof. We will show how  the latter argument yields similar asymptotics for $N(\xx, T, A)$ in \S\ref{sec:lattices}.
\subsection{Lattices} Theorem~\ref{theorem:approx} is closely related to a general theorem about approximation of directions by lattice vectors. Fix $d \geq 1, c > 0$ and define the set \begin{equation*}
R := \left\{{\bf v} = \begin{pmatrix}{\bf v}_1\\ v_2 \end{pmatrix} \in \bb R^{d} \times \bb R~:~\|{\bf v}_1\|^d|v_2| \leq c\right\},  
\end{equation*} 
which, for $v_2$ large enough, we may regard as a thinning region around the $v_2$-axis.  And, for $T>1$, identify pieces of $R$: \begin{equation}\label{cone}
P_T := \left\{{\bf v} = \begin{pmatrix}{\bf v}_1\\ v_2 \end{pmatrix} \in \bb R^{d} \times \bb R~:~\|{\bf v}_1\|^d v_2  \le c, 1 < v_2 \le T\right\} \subset \R^{d+1}. 
\end{equation} 
For a subset $A$ of $\bb S^{d-1}$, we define the subset $P_{A, T}$ of $P_T$ by 
$$P_{A, T} := \left\{ {\bf v} = \begin{pmatrix} \vv_1\\ v_2 \end{pmatrix} \in P_T: \frac{\vv_1}{\|\vv_1\|} \in A\right\}.$$ 
\begin{figure}\caption{The region $P_{50}$ in $\R^{2+1}$}
\includegraphics[width = 0.3\textwidth]{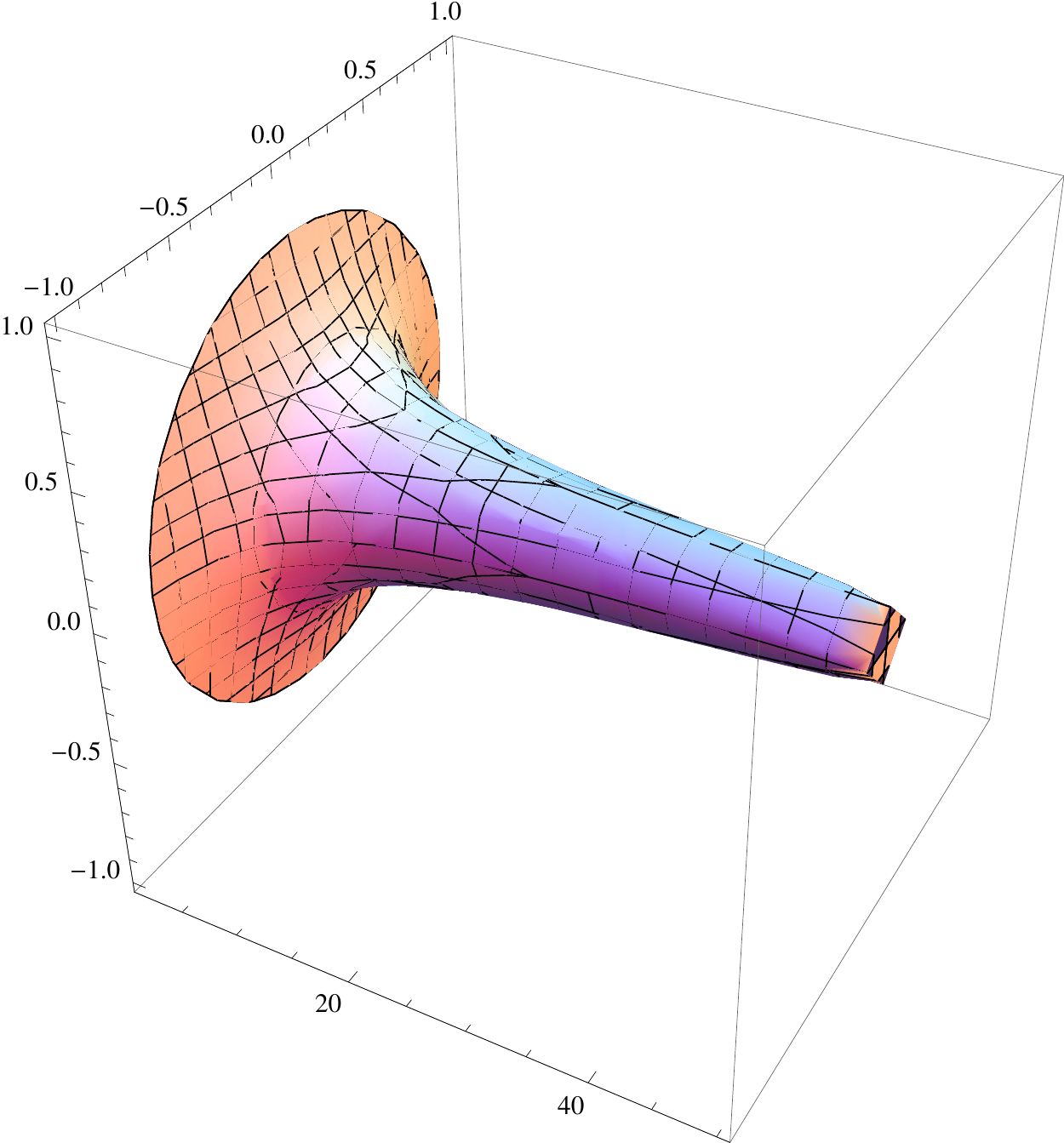}
\end{figure}
For $\La \subset \R^{d+1}$, a unimodular lattice, define
$$N(\La, T) := \# (\La \cap P_T) \text{ and } N(\La, T, A) := \# (\La \cap P_{A, T}).$$ 
Recall that $X_{d+1} := \SL_{d+1}(\R)/\SL_{d+1}(\Z)$ is the moduli space of unimodular lattices in $\R^{d+1}$ via the identification $g \SL_{d+1}(\Z) \mapsto g \Z^{d+1}$. With this identification, we endow $X_{d+1}$ with the probability measure $\mu = \mu_{d+1}$ induced by the Haar measure on $\SL_{d+1}(\R)$.
Our second main theorem is
\begin{Theorem}\label{mainbirkhoff} For $\mu$-almost every $\La \in X_{d+1}$, \begin{equation}\label{eq:latticespiral}\lim_{T \rightarrow \infty} \frac{N(\La, T, A)}{N(\La, T)} = \vol(A).\end{equation}
\end{Theorem}
\noindent The relationship to Theorem~\ref{theorem:approx} is given by the following standard construction: given ${\bf x} \in \R^d$, we form the matrix $$h_{\xx} = \begin{pmatrix} \Id_{d} & {\bf x}\\0 & 1 \end{pmatrix}$$ and the associated unimodular lattice in $\R^{d+1}$
\begin{equation}\nonumber
\Lambda_{\bf x} := h_{\xx} \bb Z^{d+1} = \left\{\begin{pmatrix} q{\bf x} - {\bf p}\\ q \end{pmatrix} ~:~{\bf p} \in \bb Z^{d}, q \in \bb Z\right\}.
\end{equation}
\noindent Then we can view the approximates $({\bf p}, q)$ of ${\bf x}$ satisfying Dirichlet's theorem with $1 \le q \le T$ as points of the lattice $\Lambda_{\bf x}$ in the regions $P_T$, and those with $\theta(\pp, q) \in A$ as points of $\La_{\xx}$ in $P_{A, T}$ with $c = C_d$. That is, we have that $N(\xx, T) = N(\La_{\xx}, T)$, and $N(\xx, T, A) = N(\La_{\xx}, T, A)$,  and, so, Theorem~\ref{theorem:approx} can be reformulated as saying that (\ref{eq:latticespiral}) holds for $\La_{\xx}$ for almost every $\xx \in \R^d$.

\subsection{Average Spiraling}
We also have an $L^1$ (average) spiraling result on the space of lattices. Fix $0 \leq \epsilon < 1$ and $T > 0$, and define 

\begin{equation}\label{defR1}
R_{\epsilon, T} := \left\{ {\bf v} \in  R~:~ \epsilon T \le v_2 \le T \right\}  
\end{equation}

\noindent and, for a subset $A$ of $\bb S^{d-1}$ with zero measure boundary,

\begin{equation}\label{defR2}
R_{A, \epsilon, T} := \left\{ {\bf v} \in R_{\epsilon, T}~:~ \frac{{\bf v}_1}{\|{\bf v}_1\|} \in A \right\}.
\end{equation} 

\noindent For a unimodular lattice $\La$, define 

$$N(\Lambda, \epsilon, T) = \#\{\Lambda \cap R_{\epsilon, T}\}$$ and 

$$N(\Lambda, A, \epsilon, T) = \#\{\Lambda \cap R_{A, \epsilon, T}\}.$$

\noindent Let $dk$ denote the Haar measure on $K := K_{d+1}:=\SO_{d+1}(\bb R)$.  Our third main theorem is

\begin{Theorem}\label{main}
For every lattice $\Lambda \in X_{d+1}$, subset $A \subset \bb S^{d-1}$ with zero measure boundary, and $\epsilon > 0$, we have that 
\begin{equation}\label{main-1}
 \lim_{T \rightarrow \infty} \frac{\int_{K} N(k^{-1}\Lambda, A, \epsilon, T)~\wrt k}{\int_{K} N(k^{-1}\Lambda, \epsilon, T)~\wrt k}= \vol(A). 
\end{equation}
\end{Theorem}
\medskip

\noindent Theorem \ref{main} is derived from our result on \emph{spherical averages of Siegel transforms}, Theorem \ref{theorem:siegel:equidist}, which we believe to be of independent interest.


\subsection{Biased Spiraling}

On the other hand, we construct explicit examples of lattices $\La$ and directions $\vv$ for which (non-averaged) equidistribution does not hold.  Our fourth main theorem is

\begin{Theorem}\label{theomaineg}
Let $d \geq 1$.  There exists a lattice $\Lambda \in \SL_{d+1}(\bb R) / \SL_{d+1}(\bb Z)$, a set $A \subset \bb S^{d-1}$ with zero measure boundary, and a sequence $\{T_n\}$ for which \begin{eqnarray*} \label{eqnmaineg}
\lim_{n \rightarrow \infty}  \frac {N(\Lambda, A, \epsilon, T_n)}{N(\Lambda, \epsilon, T_n)} \neq \vol(A)
\end{eqnarray*}
for every $1 > \epsilon\geq0$.
\end{Theorem}

\noindent For $d =1$, note that $\bb S^0:=\{-1, 1\}$ and we define $\vol(\{-1\})=\vol(\{1\}) = 1/2$.  

\medskip

\subsection*{Organization of the paper}   In \S\ref{sec:lattices}, we state Theorem \ref{theorem:siegel:equidist}, our result on spherical averages of Siegel transforms and use it to prove Theorem \ref{main}. We also prove Theorem \ref{mainbirkhoff}. In \S \ref{secSiegelEquidisProof}, we prove Theorem \ref{theorem:siegel:equidist}. The construction of examples of nonuniform spiraling for certain lattice approximates (and the proof of Theorem~\ref{theomaineg}) is carried out in \S\ref{sec:nonuniform}. Despite the extensive literature on Dirichlet's theorem and its variants, as far as we are aware, our work is the first to study the problem of spiralling for lattice approximates. In a sequel \cite{AGT2}, we prove multiparameter versions of the main theorems considered in the present work and related problems in Diophantine approximation and in \cite{AGT3} we establish versions of the main theorems in the wider generality of number fields.
\medskip

\noindent\textbf{Acknowledgements:}
This work was initiated during a visit by A.~Ghosh to the University of Illinois at Urbana-Champaign. He thanks the department for its hospitality. J.~S.~Athreya would like to thank Yale University for its hospitality in the 2012-13 academic year, when this work was completed. He would also like to thank G.~Margulis for useful discussions.  The authors would like to thank Jens Marklof for pointing us to \cite{MS3} and J.~Tseng would, in addition, like to thank Jens for useful discussions and comments. We also thank D.~Kleinbock for helpful discussions and the referee for a helpful report. 

\section{Equidistribution on the space of lattices}\label{sec:lattices}

In this section, we show how to reduce the proofs of Theorems~\ref{main} and~\ref{mainbirkhoff} to equidistribution problems on the space of lattices. Theorem~\ref{mainbirkhoff} is a consequence of this reduction and the Birkhoff ergodic theorem, which gives us almost everywhere equidistribution of trajectories for diagonal flows. For Theorem~\ref{main}, our ergodic tool Theorem~\ref{theorem:siegel:equidist} will be of independent interest, as it gives an equidistribution theorem for spherical averages of Siegel transforms for quite general functions.

Recall the definition of the Siegel transform: given a lattice $\Lambda$ in ${\bb R}^{d+1}$ and a bounded Riemann-integrable function $f$ with compact support on $\bb R^{d+1}$, denote by $\widehat{f}$ its \emph{Siegel transform}\footnote{One could define the Siegel transform only over primitive lattice points, in which case results analogous to Theorems~\ref{theorem:siegel:equidist} and~\ref{theorem:siegel:equidistUpper} also hold (using, essentially, the same proof).}:

$$ \widehat{f}(\Lambda) := \sum_{\bf v \in \Lambda \backslash \{\boldsymbol{0}\}} f(\bf v).$$

\noindent Let $\mu = \mu_{d+1}$ be the probability measure on $X_{d+1} :=  \SL_{d+1}(\bb R)/\SL_{d+1}(\bb Z)$ induced by the Haar measure on $\SL_{d+1}(\R)$ and $\wrt {\bf v}$ denote the usual volume measure on ${\bb R}^{d+1}$.  (We also let $\vol := \vol_{\bb R^{d+1}}$ denote this volume measure and will make use of the subscript should the need to distinguish it from $\vol_{\bb S^{d-1}}$ arise.) We recall the classical Siegel Mean Value Theorem \cite{Siegel}:

\begin{Theorem}
Let $f$ be as above.\footnote{This condition can be generalized to $f \in L^1(\bb R^{d+1})$.} Then $\widehat{f} \in L^{1}(X_{d+1}, \mu)$ and
$$ \int_{{\bb R}^{d+1}} f ~\wrt {\bf v} = \int_{X_{d+1}} \widehat{f} ~\wrt \mu.$$
\end{Theorem}

\medskip 

\noindent Note that if $f$ is the indicator function of a set $A \backslash \{\boldsymbol{0}\}$, then $\hat{f}(\Lambda)$ is simply the number of points in $\Lambda \cap (A \backslash \{\boldsymbol{0}\})$. 

Let $$g_t := \begin{pmatrix}e^{t}\Id_{d} & 0\\ 0 & e^{-dt} \end{pmatrix} \in \SL_{d+1}(\bb R)$$ and $e_1, \cdots, e_{d+1}$ be the standard basis of $\bb R^{d+1}$.

Note that if we set $t$ so that $e^{dt} = T$, we have
$$g_t R_{\epsilon, T} = R_{\epsilon, 1} =: R_{\epsilon}$$
\noindent and
$$ g_t R_{A, \epsilon, T} = R_{A, \epsilon, 1} =: R_{A, \epsilon}.$$

\noindent We write $\mathbf{1}_{A, \epsilon, T}$ for the characteristic function of $R_{A, \epsilon, T}$ and $\mathbf 1_{\epsilon, T}$ for the characteristic function of $R_{\epsilon, T}$, and drop the subscript when $T=1$.  

In view of the above discussions,  we have

\begin{equation}\label{eq:gt} N(\La, A, \epsilon, e^t) = \widehat{{\bf 1}}_{R_{A, \epsilon}}(g_{t/d}  \Lambda) \end{equation}

\begin{equation}\label{eq:NK} N(k^{-1}\La, A, \epsilon, T)  =  \widehat{{\bf 1}}_{R_{A, \epsilon}}(g_t k \Lambda). \end{equation} 

\noindent Integrating these formulas with respect to $t$ and $k$ respectively, we obtain

\begin{equation}\label{eq:tavg} \int_{0}^S N(\La, A, \epsilon, e^t) dt = \int_0^S \widehat{{\bf 1}}_{R_{A, \epsilon}}(g_{t/d}  \Lambda) dt\end{equation}

\begin{equation}\label{eq:Kavg} \int_K N(k^{-1}\La, A, \epsilon, T)  dk =  \int_K \widehat{{\bf 1}}_{R_{A, \epsilon}}(g_t k \Lambda) dk.\end{equation}

\subsection{Proof of Theorem~\ref{mainbirkhoff} and Theorem~\ref{theorem:approx}.}\label{sec:time} 
Let $s:= \frac{\log 2} d$. (Note that $s$ is a fixed constant).  Moore's ergodicity theorem (see, for example,~\cite{BekkaMayer}) states that the action of $g_t$ on $X_{d+1}$ is ergodic, so, by the Birkhoff ergodic theorem, for any $h \in L^1(X_{d+1}, \mu)$, we have, for almost every $\La \in X_{d+1}$,  $$\lim_{N \rightarrow \infty}\frac 1 N \sum_{n=0}^{N-1} h(g_{s}^n \La)  = \int_{X_{d+1}} h d\mu.$$ 
By Siegel's mean value formula, the functions $\widehat{{\bf 1}}_{P_{A, 2}}$ and $\widehat{{\bf 1}}_{P_{2}}$ are in $L^1(\mu)$. Write $Q_{i} = P_{2^i} \backslash P_{2^{i-1}}$, $Q_{A, i} = P_{A, 2^i} \backslash P_{A, 2^{i-1}}$. 
Then, since 
$$g_{- s} Q_i = Q_{i+1} \text{ and }   g_{- s} Q_{A, i} = Q_{A, i+1},$$we have 
 $$\sum_{i=0}^{N-1} \widehat{{\bf 1}}_{P_{2}}(g_{s}^i \La) = \sum_{i=0}^{N-1} \# (\La \cap Q_{i+1}) = \#(\La \cap P_{2^N}) = N(\La, 2^N)$$  
 $$\sum_{i=0}^{N-1} \widehat{{\bf 1}}_{P_{A, 2}}(g_{s}^i \La) = \sum_{i=0}^{N-1} \# (\La \cap Q_{A, i+1}) = \#(\La \cap P_{A, 2^N}) = N(\La, 2^N, A).$$ 
 By applying the Birkhoff ergodic theorem and the Siegel mean value theorem to these functions, we obtain, for almost every $\La$, $$\lim_{N \rightarrow \infty} \frac{1}{N} N(\La, 2^N) = \vol(P_2)  \text{ and } \lim_{N \rightarrow \infty} \frac{1}{N} N(\La, 2^N, A) = \vol(P_{A, 2}).$$ Note that if $F: [0, \infty) \rightarrow [0, \infty)$ is an increasing function, and $ \frac{F(2^k)}{k} \rightarrow \log 2$, then $$\lim_{T \rightarrow \infty} \frac{F(T)}{\log T} = 1.$$ Thus, we obtain (for almost every $\La$) \[\lim_{T \rightarrow \infty} \frac{1}{\log T} N(\La, T, A) \bigg{/} \frac{1}{\log T} N(\La, T) = \frac{\vol(P_{A, 2})}{\vol(P_2)}.\] Since $\vol(P_{A, 2})/\vol(P_2) = \vol(A)$, we obtain Theorem~\ref{mainbirkhoff}.\qed\medskip

This argument was used in~\cite{APT} to obtain logarithmic (in $T$) asymptotics for $N(\La, T)$ (and other related functions), and, as shown, also yields logarithmic asymptotics for $N(\La, T, A)$. To prove Theorem~\ref{theorem:approx}, consider the collection of matrices $\{h_{\xx}: \xx \in \R^{d}\}$ forms the \emph{horospherical} subgroup for $\{g_t\}$, and as such, the set of lattices $\{\La_{\xx}: \xx \in \R^d\}$ is the \emph{unstable} manifold for the action of $\{g_{t}\}_{t \geq 0}$ on $X_{d+1}$. In particular, for almost every $\xx \in \R^d$, $\La_{\xx}$ is Birkhoff generic for the action of $g_{s}$, which following the above argument, yields Theorem~\ref{theorem:approx} (see~\cite[Section~2.5.1]{APT}).\qed\medskip

\subsection{Statement of results for Siegel transforms}\label{sec:siegel:equidist} By (\ref{eq:Kavg}), to prove Theorem~\ref{main}, we need to show the equidistribution of the Siegel transforms of the sets $R_{A, \epsilon}$ and $R_{\epsilon}$ with respect to the integrals over $g_t$-translates of $K$.  The main ergodic tool in this setting is our fifth main theorem, a result on the spherical averages of Siegel transforms:

\begin{Theorem}\label{theorem:siegel:equidist}
Let $f$ be a bounded Riemann-integrable function of compact support on $\R^{d+1}$.  Then for any $\La \in X_{d+1}$, $$ \lim_{t \to \infty} \int_{K_{d+1}} \widehat{f}(g_t k \La) ~\wrt k = \int_{X_{d+1}}\widehat{f}~\wrt \mu.$$

\end{Theorem} 

\noindent We note that the above theorem is reminiscent of~\cite[Theorem 3.4]{EMM} of Eskin-Margulis-Mozes, but the compact group and the one-parameter diagonal subgroup used there are different. In fact, a proof of this theorem in this spirit can be assembled from the work of Kleinbock-Margulis~\cite[Appendix]{KM}, but we present an elementary proof relying on counting lattice points in balls. This  proof is adapted from~\cite[Section~5.1]{MS3}, where the result is proved for balls around the origin (for the slightly different context of a cut-and-project quasicrystal).\footnote{For an introduction to~\cite{MS3}, see~\cite{MS1}.  Also see~\cite{MS2} and~\cite{MS4}.}  For the proof of Theorem~\ref{theorem:siegel:equidist}, we must adapt this proof for balls not containing the origin.  This is done in Section~\ref{secSiegelEquidisProof} for the upper bound:

\begin{Theorem}\label{theorem:siegel:equidistUpper}
Let $f$ be a bounded function of compact support in  $\R^{d+1}$ whose set of discontinuities has zero Lebesgue measure.  Then for any $\La \in X_{d+1}$, $$ \lim_{t \to \infty} \int_{K_{d+1}} \widehat{f}(g_t k \La) ~\wrt k \leq \int_{X_{d+1}}\widehat{f}~\wrt \mu.$$ 
\end{Theorem} 

\begin{rema}
Since $\RR^{d+1}$ is $\sigma$-compact, it follows immediately from the theorem that the assumption that $f$ has compact support can be replaced with that of $f \in  L^1(\bb R^{d+1})$---the other assumptions are still, however, necessary for the proof.
\end{rema}

\begin{coro}
Let $f$ be a bounded Riemann-integrable function of compact support in  $\R^{d+1}$.  Then for any $\La \in X_{d+1}$, $$ \lim_{t \to \infty} \int_{K_{d+1}} \widehat{f}(g_t k \La) ~\wrt k \leq \int_{X_{d+1}}\widehat{f}~\wrt \mu.$$ 
\end{coro}

\begin{proof}
Immediate from the theorem and the Lebesgue criterion.
\end{proof}

In this paper, we will focus on upper bounds, i.e. on the proof of Theorem \ref{theorem:siegel:equidistUpper}. The lower bound, on the other hand, follows either from the methods in \cite{KleinMarg} or by applying the following  equidistribution theorem (Theorem~\ref{theorem:DRS}) of Duke, Rudnick and Sarnak (cf. \cite{DRS}) (a simpler proof was given by Eskin and McMullen \cite{EMc} using mixing and generalized by Shah \cite{Shah}) and then approximating the Siegel transform $\widehat{f}$ from below by $h \in C_c(X_{d+1})$. 
\begin{Theorem}\label{theorem:DRS}
Let $G$ be a non-compact semisimple Lie group and let $K$ be a maximal compact subgroup of $G$. Let $\Gamma$ be a lattice in $G$, let $\lambda$ be the probabilty Haar measure on $G/\Gamma$, and let $\nu$ be any probability measure on $K$ which is absolutely continuous with respect to a Haar measure on $K$. Let $\{a_n\}$ be a sequence of elements of $G$ without accumulation points. Then for any $x \in G/\Gamma$ and any $h \in C_{c}(G/\Gamma)$,
$$ \lim_{n \to \infty} \int_{K} h(a_n k x)~\wrt\nu(k) = \int_{G/\Gamma}h~\wrt \lambda.$$ 
\end{Theorem} 

\begin{rema}
One can replace $\wrt k$ by $\wrt\nu(k)$ in Theorems~\ref{theorem:siegel:equidist}~and~\ref{theorem:siegel:equidistUpper} without any changes to the proofs.

\end{rema}

\subsection{Proof of Theorem~\ref{main}}  We prove Theorem~\ref{main} using Theorem~\ref{theorem:siegel:equidist}, while deferring the proof of the latter to Section~\ref{secSiegelEquidisProof}.  Thus, applying Theorem~\ref{theorem:siegel:equidist} to characteristic functions of $R_{A, \epsilon}$ and $R_{\epsilon}$, we obtain

$$ \lim_{t \to \infty} \int_{K} \widehat{{\bf 1}}_{R_{A, \epsilon}}(g_t k \Lambda)\wrt \nu(k) =\int_{X_{d+1}} \widehat{{\bf 1}}_{R_{A, \epsilon}}\wrt \mu = \vol(R_{A, \epsilon}),  $$

\noindent where we have applied Siegel's mean value theorem in the last equality.\footnote{There are two parameters $0<r_1 < r_2$, easy to compute and depending on $\varepsilon$, such that, if we define the set $W$ to be the union of all intervals with a terminal point on the sphere of radius $r_1$ in $\RR^d \times \{0\}$, the other terminal point on the sphere of radius $r_2$ in $\RR^d \times \{0\}$, lying in a ray emanating from the origin of $\RR^d \times \{0\}$, and passing through a point of $A$, then we may regard the set $R_{A, \epsilon}$ as the region under the graph of the continuous function $v_2 = \frac{c}{\|\boldsymbol{v}_1\|^d}$ over $W$ union a cylinder over the sphere of radius $r_1$.  Since $\partial(A)$ has zero Lebesgue measure (on $\s^{d-1}$), so does $\partial(W)$ (on $\RR^d$) using polar coordinates for Lebesgue measurable functions.  As $R_{A, \epsilon}$ is bounded, it has finite Lebesgue measure.  Applying Fubini's Theorem and noting that the cylinder over the sphere is a Jordan set, one sees that the compact set  $\partial(R_{A, \epsilon})$ has zero Lebesgue measure (on $\RR^{d+1}$).  It is easy to see that a compact set of zero Lebesgue measure also has zero Jordan content.  Therefore, the set $R_{A, \epsilon}$ is a Jordan set.  Likewise, for $R_{\epsilon}$.   It is well-known that the characteristic functions of Jordan sets are Riemann-integrable and, thus, Theorem~\ref{theorem:siegel:equidist} applies to these functions.}  We apply this to numerator as well as denominator in (\ref{main-1}) to get

$$ \lim_{T \to \infty} \frac{\int_{K}N(k^{-1}\La, A, \epsilon, T)~\wrt k}{\int_{K} N(k^{-1}\La, \epsilon, T)~\wrt k}  = \frac{\vol(R_{A, \epsilon})}{\vol(R_{\epsilon})} = \vol(A), $$

\noindent which finishes the proof.\qed\medskip

\section{Proof of Theorem~\ref{theorem:siegel:equidistUpper}}\label{secSiegelEquidisProof}  As mentioned, to prove Theorem~\ref{theorem:siegel:equidist}, we need only show the upper bound (Theorem~\ref{theorem:siegel:equidistUpper}):  \begin{equation}
 \lim_{t \to \infty} \int_{K_{d+1}} \widehat{f}(g_t k \La) ~\wrt k \leq \int_{X_{d+1}}\widehat{f}~\wrt \mu.  
\end{equation}
We will approximate using step functions on balls (see Section~\ref{secApproxSimpleFcts}), where we use the norm on $\R^{d+1} = \R^d \times \R$ given by the supremum of Euclidean norm in $\R^d = \mbox{span}(e_1, \cdots, e_d)$ and by absolute value in $\R = \mbox{span}(e_{d+1})$.  Hence, balls will be open regions of $\RR^{d+1}$ that are rods (i.e. solid cylinders).  We need four cases:  balls centered at $\boldsymbol{0} \in \RR^{d+1}$, balls centered in $\spn(e_{d+1}) \backslash\{\boldsymbol{0}\}$, balls centered in $\spn(e_1, \cdots, e_d) \backslash\{\boldsymbol{0}\}$, and all other balls.  Since we will approximate using step functions, it suffices (as we show in Section~\ref{secApproxSimpleFcts}) to assume that the balls in the second case do not meet $\boldsymbol{0}$ and in the last case do not meet $\spn(e_{d+1}) \cup \spn(e_1, \cdots, e_d)$.\footnote{We note that the second and the fourth cases already suffice to show Theorem~\ref{main}.}    Let $E := B(\boldsymbol{w}, r)$ be any such ball and $\chi_E$ be its characteristic function.  By the monotone convergence theorem, we have \[\int_{K_{d+1}} \widehat{\chi}_E(g_t k \La) ~\wrt k =\sum_{\boldsymbol{v} \in \La \backslash \{\boldsymbol{0}\} }\int_{K_{d+1}} \chi_{k^{-1} g_t^{-1} E}(\boldsymbol{v}) ~\wrt k.\]  

We show each case in turn.  For the first case, we refer the reader to~\cite[Section~5.1]{MS3}, in which the desired result (with balls given by the Euclidean norm in $e_1, \cdots, e_{d+1}$) is shown for quasicrystals and is essentially the same for us.  (Alternatively, a simplified version of the proof of the third case will also show the first case.)  The proofs of the other three cases adapt this basic idea.  For convenience of exposition, we show the fourth case before the third.

\subsection{The second case:  balls centered in $\spn(e_{d+1}) \backslash\{\boldsymbol{0}\}$}  In this case, $\boldsymbol{w} = w  e_{d+1}$ for some $w \neq 0$.  The proof is similar for whether $w$ is positive or negative, so we may assume without loss of generality that $w >0$.  Let \[\widetilde{B}^d := \widetilde{B}^d(r):=\{(x_1, \cdots, x_d)^t \mid {R^2}{x_1^2} + \cdots +{R^2}{x_d^2} < r^2\}\] and \[I(w) := I(w,r):= \{x_{d+1} \mid w-r < \frac{x_{d+1}}{R^d} < w+r \}\]  where $R := e^t$.  Then $g_t^{-1} E$ is the rod given by \[\widetilde{B}^d \times I(w).\]  Replacing the $< r^2$ with $= \rho^2$ for a $0\leq \rho^2<r^2$ and using the equation for the sphere $\tau \s^d$ where $\tau >0$, we note that the intersection has at most two values for the $x_{d+1}$.  Moreover, since we only need to consider balls not meeting $\boldsymbol{0}$, we may assume that the rod completely lies in the half-space determined by $e_{d+1}$ and thus $x_{d+1} = \sqrt{\tau^2 - \rho^2/R^2}=:c(\rho)$ is our value of intersection provided that it lies in $I(w)$.  Therefore, the intersection $\mathfrak{C}(\tau)$ of the rod with the sphere is a $d$-dimensional \textit{cap,} through which each intersection (i.e. slice) with the affine hyperplane through and normal to $c(\rho) e_{d+1}$ is a $d-1$-dimensional sphere (with radius $\rho/R$).  There are two types of caps.  A \textit{full cap} is a cap such that \begin{equation}\label{eqnCaps}c(\rho) \vert_{0 < \rho^2<r^2} \subset I(w).\end{equation}  We remark that the continuity of the function $c$ means that (\ref{eqnCaps}) is interval inclusion.  Otherwise, a cap is called an \textit{end cap} because it is near one or the other end of the rod.  Now let $\tau_-:= R^d (w - r)$ and $\tilde{\tau}_+ := R^d (w + r)$.  Then, for $0< \tau < \tau_-$ and $\tau_+ :=  \sqrt{\tilde{\tau}_+^2 + r^2/R^2}< \tau$, the intersection between sphere and rod is empty.  Since whether a cap is full or end depends only on (\ref{eqnCaps}), $\mathfrak{C}(\tau)$ are all full caps for $\tilde{\tau}_-:= \sqrt{\tau_-^2 + r^2/R^2}< \tau < \tilde{\tau}_+$.

We are only interested in the case where $R$ is large (and where $r$ is small).  Since \begin{align*} \tilde{\tau}_- - \tau_- & = O(R^{-(d+2)}) \\ \tau_+ - \tilde{\tau}_+ &= O(R^{-(d+2)}),\end{align*} the end caps are negligible (as we shall see below).  We remark that these estimates also give the approximate ``depth'' of any full cap.  

For $R$ large, the $d$-dimensional volume of a full cap is nearly, but slightly larger than, that of the $d$-dimensional ball $\mathfrak{B}$ with boundary sphere exactly the slice through $c(r)$.  More precisely, \begin{align}\label{eqnCapSliceRatio}\frac{\vol_{\tau\s^d}(\mathfrak{{C}(\tau))}}{\vol_d(\mathfrak{B})}  \searrow \gamma(\tilde{\tau}_+) \geq 1\end{align} as $\tau \nearrow \tilde{\tau}_+$ over the interval $(\tilde{\tau}_-,\tilde{\tau}_+)$.  (Note that $\gamma(\tilde{\tau}_+) \searrow 1$ as $\tilde{\tau}_+ \nearrow \infty.$)

Let $\tilde{K}_d:= \begin{pmatrix}\SO_d(\bb R) & 0\\ 0 & 1 \end{pmatrix}$.  Then $g_t^{-1} E$ is stabilized by every element of $\tilde{K}_d$.   It is well-known that the unit sphere $\s^d$ can be realized as the homogeneous space $  \tilde{K}_d \backslash K_{d+1}$.  Consequently, \begin{align}\label{eqnInvarofRotationonSphere}\int_{K_{d+1}} \chi_{k^{-1} g_t^{-1} E}(\boldsymbol{v}) ~\wrt k & = \frac 1 {\vol_{\s^d}(\s^d)}\int_{\s^d} \chi_{s^{-1} g_t^{-1} E}(\boldsymbol{v}) ~\wrt \vol_{\s^d}(s) \nonumber\\ & =\frac {\vol_{\s^d}(\s^d \cap \|\boldsymbol{v}\|^{-1}g_t^{-1} E)}{\vol_{\s^d}(\s^d)} =: A_R^E(\|\boldsymbol{v}\|)\end{align} where the second equality follows from the correspondence $s\boldsymbol{v} \in g_t^{-1} E \iff s \boldsymbol{v}/\|\boldsymbol{v}\| \in \|\boldsymbol{v}\|^{-1}g_t^{-1} E \cap \s^d.$  The invariance of the ratio of the volume measures on spheres of different radii implies that \[A_R^E(\tau) = \frac {\vol_{\tau \s^d}(\tau \s^d \cap g_t^{-1} E)}{\vol_{\tau\s^d}(\tau\s^d)}\] for $\tau >0$.  By (\ref{eqnCapSliceRatio}) and the fact that the full caps are determined by polynomial equations,  $A_R^E(\tau)$ is a strictly decreasing smooth function with respect to $\tau$ over the interval $(\tilde{\tau}_-,\tilde{\tau}_+)$.  

To take care of the end caps, we shall replace our rod with a slightly larger one.  More precisely, replace $I(w,r)$ with $I(w, r+ \frac1 {R^{2d+1}})$, which, for $R$ large enough, lengthens the rod enough (see estimates above) so that all caps of the shorter rod are now full caps of the longer rod.  In particular, $A_R^E(\tau)$ is strictly decreasing over the interval $(\tau_-,\tau_+)$.

Let $B_{\Euc}(\boldsymbol{0}, \tau)$ denote a ball of radius $\tau$ in $\RR^{d+1}$ with respect to the Euclidean norm.  Now it follows from the formula for $A_R^E$ that \[\sum_{\boldsymbol{v} \in \La \backslash \{\boldsymbol{0}\} } A_R^E(\|\boldsymbol{v}\|) \leq \int_{\tau_-}^{\tau_+} \#\big(B_{\Euc}(\boldsymbol{0}, \tau) \cap\La \backslash \{\boldsymbol{0}\}\big)   ~(-\wrt A_R^E(\tau))\] where the integral is the Riemann-Stieltjes integral.  We remark that the integrability of the function $\#\big(B_{\Euc}(\boldsymbol{0}, \tau) \cap\La \backslash \{\boldsymbol{0}\}\big)$ follows from its monotonicity.  

Now the well-known generalizations of the Gauss circle problem (or, alternatively,~\cite[Proposition~3.2]{MS3}) show that \begin{align}\label{eqnGaussCircle}\#\big(B_{\Euc}(\boldsymbol{0}, \tau) \cap\La \backslash \{\boldsymbol{0}\}\big) & \leq (1+ \varepsilon)  \vol(B_{\Euc}(\boldsymbol{0}, \tau))\nonumber \\ &= (1+ \varepsilon)  \vol(B_{\Euc}(\boldsymbol{0}, 1)) \tau^{d+1} \end{align} for $\varepsilon \rightarrow 0$ as $R \rightarrow \infty$ (and hence as $\tau_- \rightarrow \infty$).  

Also, by (\ref{eqnCapSliceRatio}), we have that $C(\tau_-):=\vol_{\tau\s^d}(\mathfrak{{C}(\tau_-))} \rightarrow \vol_d(\mathfrak{B})$ as $R \rightarrow \infty$ (and hence as $\tau_- \rightarrow \infty$).  And we have \[A_R^E(\tau) \leq \frac{C(\tau_-)}{\tau^d{\vol_{\s^d}(\s^d)}}\]  over the interval $(\tau_-,\tau_+)$.  Consequently, it follows that \begin{align*}\sum_{\boldsymbol{v} \in \La \backslash \{\boldsymbol{0}\} } A_R^E(\|\boldsymbol{v}\|) &\leq d \int_{\tau_-}^{\tau_+}  (1+ \varepsilon)  \vol(B_{\Euc}(\boldsymbol{0}, 1))\frac{C(\tau_-)}{{\vol_{\s^d}(\s^d)}} ~\wrt \tau \\ & = d (1 + \varepsilon) \frac {\vol(B_{\Euc}(\boldsymbol{0}, 1))}{{\vol_{\s^d}(\s^d)}} C(\tau_-) (\tau_+ - \tau_-). \end{align*}  Now $C(\tau_-) (\tau_+ - \tau_-)$ is the volume of a rod that has length within $O(\frac 1 {R^{d+1}})$ of the length of our original rod, but with cross-section volume $C(\tau_-)$.  Let $R \rightarrow \infty$ and $\varepsilon \rightarrow 0$, we have \begin{align*}\sum_{\boldsymbol{v} \in \La \backslash \{\boldsymbol{0}\} } A_R^E(\|\boldsymbol{v}\|) &\leq \vol(E),\end{align*} as desired.  (Note that $ \frac {\vol(B_{\Euc}(\boldsymbol{0}, 1))}{{\vol_{\s^d}(\s^d)}}= \frac 1 {d+1}.$)

\subsection{The fourth case:  all other balls}  The proof of this case is similar to the second case.  The differences are as follows.  In this case, $g_t^{-1} E$ is no longer invariant under every element of $\widetilde{K}_d$; however, we need this property only to show (\ref{eqnInvarofRotationonSphere}), which also follows because $\wrt vol_{\s^d}(s)$ is preserved under rotations.  Caps are no longer such simple geometric objects (their boundaries are ellipsoids, not spheres); however, the proof is unaffected by this change.

\subsection{The third case:  balls centered in $\spn(e_1, \cdots, e_d) \backslash\{\boldsymbol{0}\}$}  The difference between this case and the fourth case is that the ends of the rod do not both go to $w e_{d+1}$ for $w \rightarrow \infty$ or both for $w \rightarrow -\infty$, but one end goes to one direction and the other goes to the other.  To take care of this issue, we must consider what happens near the origin. Fix the lattice $\Lambda$.  Then, by the discreteness of the lattice, there is a ball $B_{\Euc}(\boldsymbol{0}, \tau_0))$ that does not meet $\Lambda \backslash \{\boldsymbol{0}\}$ for some $\tau_0>0$ depending only on $\Lambda$.  Since there are no relevant lattice points in any ball around the origin of radius smaller than $\tau_0$, we need only consider $\tau \geq \tau_0$.  Note that the rod and sphere meet in two connected components, each of which we will refer to as caps.  This fact, however, does not affect the proof, except in minor ways as noted below.

Now recall the center of our ball $E$ is $\boldsymbol{w}$, whose last coordinate is $w_{d+1}=0$.  Thus, $g_t^{-1} \boldsymbol{w} \rightarrow \boldsymbol{0}$.  And since the rod $g_t^{-1}E$ is contracting in $x_1, \cdots, x_d$, our analysis of $A_R^E (\tau)$ for the fourth case also hold in this case for large enough $R$ over the desired range $\tau \geq \tau_0$.  In particular, for large $R$, we have that $A_R^E(\tau)$ is a smooth decreasing function of $\tau$ over $(\tau_0, \tau_+)$ such that \[A_R^E(\tau) \leq \frac{C(\tau_0)}{\tau^d{\vol_{\s^d}(\s^d)}}.\]  Here $C(\tau_0)$ is the volume of the two caps that is the relevant intersection.  We may take care of end caps as in the previous cases.

Since, for $R$ large, our analysis in the fourth case applies, we have from (\ref{eqnCapSliceRatio}) that $C(\tau_0) = O(R^{-d})$ or, more precisely, \begin{equation}\label{eqnCrossSectVolumeEst} \vol_d (\mathfrak{B}) \leq C(\tau_0) <  \frac {d+1} d \vol_d(\mathfrak{B}).\end{equation}  Here, $\vol_d(\mathfrak{B})$ is twice the volume of the intersection of a $d$-hyperplane normal to $e_{d+1}$ with the rod.  Finally, as in the second case, for every $\varepsilon >0$, there exists a (large) $\tau_1>0$ depending only on $\Lambda$ such that for all $\tau \geq \tau_1$, we have that (\ref{eqnGaussCircle}) holds.  Hence, for large $R$, we have, as in the previous cases, \begin{align*}\sum_{\boldsymbol{v} \in \La \backslash \{\boldsymbol{0}\} } A_R^E(\|\boldsymbol{v}\|)  & \leq \int_{\tau_0}^{\tau_+} \#\big(B_{\Euc}(\boldsymbol{0}, \tau) \cap\La \backslash \{\boldsymbol{0}\}\big)   ~(-\wrt A_R^E(\tau)) \\ & \leq const(\tau_1)  \int_{\tau_0}^{\tau_1}  ~(-\wrt A_R^E(\tau)) +  (1+ \varepsilon)  \vol(B_{\Euc}(\boldsymbol{0}, 1)) \int_{\tau_1}^{\tau_+} \tau^{d+1}  ~(-\wrt A_R^E(\tau)) \\ & \leq O(R^{-d}) + d (1 + \varepsilon) \frac {\vol(B_{\Euc}(\boldsymbol{0}, 1))}{{\vol_{\s^d}(\s^d)}}C(\tau_0) (\tau_+ - \tau_1),\end{align*} where the inequality for second integral follows as in previous cases.  Since the function that counts lattice points in larger and larger balls is monotonically increasing, $const(\tau_1):=(1+ \varepsilon)  \vol(B_{\Euc}(\boldsymbol{0}, 1)) \tau_1^{d+1}$, and thus the implicit constant depends only on $r$, $\varepsilon$, $\tau_0$, and $\tau_1$.  Applying (\ref{eqnCrossSectVolumeEst}) and letting $R \rightarrow \infty$ and $\varepsilon \rightarrow 0$ yields the desired result.  (Recall that $\vol_d(\mathfrak{B})$ is twice the volume of the intersection of a $d$-hyperplane normal to $e_{d+1}$ with the rod.)

\subsection{Finishing the proof}\label{secApproxSimpleFcts}  We may now approximate using step functions to obtain Theorem~\ref{theorem:siegel:equidistUpper}.\footnote{The construction presented here is general and may be of independent interest.}  For the convenience of the reader, we give a proof.  Let $F \subset \RR^{d+1}$ be a compact subset with nonempty interior (i.e. $F^\circ \neq \emptyset$) and measure zero boundary (i.e. $\vol(\partial F)=0$).  Recall that our balls, which we have called rods, are given by the Euclidean norm in $e_1, \cdots, e_d$ and by absolute value in $e_{d+1}$---the unit ball defines a norm $\|\cdot\|_{\textrm{rod}}$ on $\RR^{d+1}$.  Let us denote the rods defined in the beginning of this section (Section~\ref{secSiegelEquidisProof}) as case-one, case-two, case-three, or case-four rods, respectively---recall these rods are open sets.  We will refer to a sequence of sets as a \textit{disjoint} sequence if the sets in the sequence are pairwise disjoint.

\begin{lemm}\label{lemmVitialCoveringTheorem}  There exists a disjoint sequence $\{B_n\}_{n=1}^\infty$ of case-one, case-two, case-three, and case-four rods of $\RR^{d+1}$ so that \begin{align*} F^\circ  & \supset \coprod B_n \\ \vol(F) & = \sum \vol(B_n).\end{align*}  Moreover, the radii of all the rods in the sequence may be chosen to be $\leq \eta$ for any choice of $\eta >0$.

\end{lemm}

This lemma follows easily from a classical result:  the Vitali covering theorem (see~\cite[Theorem~1.6]{He} for example).

\begin{theo}[Vitali Covering Theorem]  Let $A$ be a subset of a doubling metric space $(X, \mu)$ and $\mathcal{F}$ be a collection of closed balls centered at $A$ such that \[\inf\{r>0 \ \vert \ \overline{B(a,r)} \in \mathcal{F}\} = 0\] for each $a \in A$.  Then there exists a countable subcollection $\{\overline{B_n}\}_{n=1}^\infty$ of pairwise disjoint closed balls such that \[\mu\bigg(A \backslash \bigcup_{n=1}^\infty \overline{B_n}\bigg)=0.\]
 
\end{theo}



It will soon become apparent that we would like to mimic the construction of the Riemann integral, which, for the functions that we are interested in, is constructed over cubes.  But, instead of cubes, we will use partitions (mod 0) consisting of pairwise disjoint open rods and will use the Vitali covering lemma (stated below).  Let $B_0$ be a large enough open rod, not necessarily of any of the four cases, containing $\supp(f)$ and let $\rho_0$ be its radius.  Let ${\Pa}_0 = \{B_0\}$ be the initial partition.  

\subsubsection{Refinements of partitions}   


We recursively define refinements as follows:  let ${\Pa}_j$ be a partition (mod 0) in which every element of $\Pa_{j-1}$ with radius $> 2^{-j} \rho_0$ is replaced with a sequence of pairwise disjoint open rods as constructed in Lemma~\ref{lemmVitialCoveringTheorem} with $\eta \leq 2^{-j} \rho_0$.\footnote{Note that the many choices that we make in constructing the partition and its refinements do not affect the proof.}   We note that $\Pa_j$ is a countable union of case-one, case-two, case-three, and case-four rods only.  Let us define the \textit{size} of a partition $\Pa_j$ to be the supremum over all radii of elements in $\Pa_j$.  Then the size of $\Pa_j \leq 2^{-j} \rho_0$.  By construction,  we have that \begin{align*}
\bigcup_{B \in \Pa_j} B \subset \bigcup_{B \in \Pa_{j-1}} B
\end{align*} and, by Lemma~\ref{lemmVitialCoveringTheorem}, we have that both sets have Lebesgue measure equal to $\vol(B_0)$ and thus the following set has Lebesgue measure zero \[\D' = \overline{B}_0 \bigg{\backslash} \bigcap_{j=0}^\infty \bigcup_{B \in \Pa_j} B.\]  Now let $\D''$ denote the set of discontinuities of $f$, a set of Lebesgue measure zero by assumption.  Let $\D := \D' \cup \D''$.  Note that $\D$ is subset of $\overline{B}_0$, but it need not be compact.

\subsubsection{Approximating $f$ by step functions}

We now wish to approximate $f$ on the full measure set $\overline{B}_0 \backslash \D$ by step functions over the rods of $\Pa_j$.  Define the step functions on $\overline{B}_0 \backslash \D$ as follows:  

\begin{align*}
f_j := \sum_{B \in \Pa_j} \big(\sup f \big{\vert}_B\big) \chi_{B \backslash \D}.
\end{align*}  For any point of $\overline{B}_0 \backslash \D$, we note that the sum only has one term.  Moreover, it is easy to see that these functions converge Lebesgue-a.e. to $f$:

\begin{lemm}\label{lemmApproxStepFct}
The step functions $f_j \rightarrow f$ for every point of $\overline{B}_0 \backslash \D$. 
\end{lemm}


\subsubsection{Proof for the case of step functions.}  We now prove Theorem~\ref{theorem:siegel:equidistUpper} for $f_j$.   Take any total ordering of the set $\Pa_j = \{B_n\}_{n=1}^\infty$ and let \[f_{j,n} := \sum_{i=1}^n \big(\sup f \big{\vert}_{B_i}\big) \chi_{B_i \backslash \D}.\]  Then we have that \[\lim_{n \rightarrow \infty} f_{j,n} = f_j,\]  where the limit denotes pointwise convergence over the domain $\overline{B}_0 \backslash \D$.

Let $M := \sup |f\vert_{\overline{B}_0}|$.  Fix a lattice $g_t k \La \backslash \{\boldsymbol{0}\}$.  Since $M \chi_{\overline{B}_0 \backslash \D}\leq M \chi_{\overline{B}_0}$, we have that \[\sum_{\boldsymbol{v} \in g_t k \La \backslash \{\boldsymbol{0}\}} M \chi_{\overline{B}_0 \backslash \D}(\boldsymbol{v} )<\infty\] because the same sum over $\chi_{\overline{B}_0}$ is the number of lattice points in this ball, which is finite.  Now since $|f_{j,n}| \leq  M \chi_{\overline{B}_0 \backslash \D}$, dominated convergence implies that \begin{align}
\sum_{\boldsymbol{v} \in g_t k \La \backslash \{\boldsymbol{0}\}} f_j = \lim_{n \rightarrow \infty} \sum_{\boldsymbol{v} \in g_t k \La \backslash \{\boldsymbol{0}\}} f_{j,n} <\infty. 
\label{eqnCountingPointsFixLattice}\end{align}  Here, more precisely, $\boldsymbol{v} \in g_t k \La \backslash \{\boldsymbol{0}\} \cap \overline{B}_0 \backslash \D$.  Furthermore, we note that (\ref{eqnCountingPointsFixLattice}) holds for every $k \in K_{d+1}$.  Applying dominated convergence again, we conclude \begin{align}\label{eqnLimitStepFct}
 \int_{K_{d+1}} \widehat{f_j}(g_t k \La) ~\wrt k =  \lim_{n \rightarrow \infty} \int_{K_{d+1}} \sum_{\boldsymbol{v} \in g_t k \La \backslash \{\boldsymbol{0}\}} f_{j,n} ~\wrt k.
\end{align}

Now define $\widetilde{f}_{j,n} := \sum_{i=1}^n \big(\sup f \big{\vert}_{B_i}\big) \chi_{B_i}$ and repeating the above with $\widetilde{f}_{j,n}+ M \chi_\D$ (which is dominated by $2 M \chi_{\overline{B}_0}$) in place of $f_{j,n}$ yields \begin{align}\label{eqnBndStepFct}
 \int_{K_{d+1}} \widehat{f_j}(g_t k \La) ~\wrt k \leq  \lim_{n \rightarrow \infty} \int_{K_{d+1}} \sum_{\boldsymbol{v} \in g_t k \La \backslash \{\boldsymbol{0}\}} \widetilde{f}_{j,n} ~\wrt k +  \int_{K_{d+1}} \sum_{\boldsymbol{v} \in g_t k \La \backslash \{\boldsymbol{0}\}} M \chi_\D ~\wrt k.
\end{align}  Note that \[\lim_{n \rightarrow \infty} \widetilde{f}_{j,n}+ M \chi_\D\] is well defined and the limit denotes pointwise convergence on $\overline{B}_0$.


To handle the integral involving the zero Lebesgue measure set $\D$, we proceed as follows.  Recall that Lebesgue measure is outer regular, which applied to $\D$ is the following: \[\vol(\D) = \inf\{\vol(U) \ \vert \ U \supset \D, \ U \textrm{ open}\}.\]  It is easy to see that the subcollection of open rods \[\mathfrak{T}:= \bigg{\{}B\left(x,\frac 1 5 r\right) \mid B(x,r) \textrm{ is a case-one, case-two, case-three, or case-four rod }\bigg{\}}\] is a basis of the usual topology of $\RR^{d+1}$.  Consequently, for every $\gamma'>0$, there exists a family $\{ B'_\alpha \} \subset \mathfrak{T}$ such that 
\begin{align*}
\D \subset \cup B'_\alpha \ \textrm{ and} \ \vol(\cup B'_\alpha) < \gamma'.
\end{align*}

We now require a classical result:  the Vitali covering lemma (see~\cite[Lemma~1.9]{Fa} or~\cite[Theorem~1.16]{He} for example).

\begin{theo}[Vitali covering lemma]\label{theoVitaliCoveringLemma}
 Let $\mathfrak{C}$ be a collection of balls contained in a bounded subset of $\RR^{d+1}$.  Then there exists a finite or countably infinite subcollection of pairwise disjoint balls $\{B_m\}$ such that \[\bigcup_{B \in \mathfrak{C}} B \subset \bigcup_m 5B_m\] where $5 B_m$ is the ball concentric with $B_m$ of $5$ times the radius.  
 
 Moreover, the subcollection $\{B_m\}$ can be chosen so that \[\sum_{m=1}^\infty \chi_{5B_m}(x) \leq C(d)\] where $C$ is a constant depending only on $d$.
\end{theo} 

The covering lemma allows us to cover $\D$ by a countable collection of our rods:  

\begin{lemm}\label{lemVitaliCovering}
For every $\gamma>0$, there exists (at most) a countable sequence $\{\widetilde{B}_m\}_{m=1}^\infty$ of case-one, case-two, case-three, and case-four rods such that \begin{align*}
\D \subset \cup_{m=1}^\infty \widetilde{B}_m \ \textrm{ and} \ \sum_{m=1}^\infty \vol(\widetilde{B}_m) < \gamma.
\end{align*}
 
\end{lemm}
\begin{proof}
Let $\{B'_\alpha \}$ (defined above) be $\mathfrak{C}$ and $\{B'_m\}_{m=1}^\infty$ be the subcollection in the Vitali covering lemma.  Let $\widetilde{B}_m := 5 B'_m$.  Then $\{\widetilde{B}_m\}$ is a countable collection of case-one, case-two, case-three, and case-four rods.  Consequently, we have \begin{align*}
\vol(\D) \leq \sum_{m=1}^\infty \vol(\widetilde{B}_m) = \sum_{m=1}^\infty 5^{d+1} \vol(B'_m) =   5^{d+1}  \vol(\cup_{m=1}^\infty B'_m) < 5^{d+1}  \gamma'.
\end{align*}  Our desired result is now immediate.
\end{proof}


Choose $\gamma>0$ small.  To bound the integral involving $\D$ in (\ref{eqnBndStepFct}), we may approximate with the $\{\widetilde{B}_m\}_{m=1}^\infty$ from Lemma~\ref{lemVitaliCovering} as follows.  We have \begin{align*}
0 \leq \int_{K_{d+1}} \sum_{\boldsymbol{v} \in g_t k \La \backslash \{\boldsymbol{0}\}} M \chi_\D ~\wrt k \leq \int_{K_{d+1}} \sum_{\boldsymbol{v} \in g_t k \La \backslash \{\boldsymbol{0}\}} \sum_{m=1}^\infty M \chi_{\widetilde{B}_m}  ~\wrt k
\end{align*} because \[M \chi_\D (\boldsymbol{v}) \leq M \chi_{\cup_m \widetilde{B}_m}(\boldsymbol{v})\leq\sum_{m=1}^\infty M \chi_{\widetilde{B}_m}(\boldsymbol{v})\] for all $\boldsymbol{v}$.  Now, replacing all instances of our use of the dominated convergence theorem by the monotone convergence theorem in the argument that we used to show (\ref{eqnLimitStepFct}), we have that \begin{align*}
\int_{K_{d+1}} \sum_{\boldsymbol{v} \in g_t k \La \backslash \{\boldsymbol{0}\}} \sum_{m=1}^\infty M \chi_{\widetilde{B}_m}  ~\wrt k =  \lim_{N \rightarrow \infty} \sum_{m=1}^N M \int_{K_{d+1}} \sum_{\boldsymbol{v} \in g_t k \La \backslash \{\boldsymbol{0}\}} \chi_{ \widetilde{B}_m}  ~\wrt k,
\end{align*} where, a priori, the limit may be infinite---we will, however, show that the limit is finite and can be made arbitrarily small below.

Putting this together with (\ref{eqnBndStepFct}) yields \begin{align}\label{eqnSplitMainandNull}
 \int_{K_{d+1}} \widehat{f_j}(g_t k \La) ~\wrt k & \leq  \lim_{n \rightarrow \infty} \sum_{i=1}^n \big(\sup f \big{\vert}_{B_i}\big) \int_{K_{d+1}} \sum_{\boldsymbol{v} \in g_t k \La \backslash \{\boldsymbol{0}\}} \chi_{B_i} ~\wrt k \\ \nonumber& +\lim_{N \rightarrow \infty} \sum_{m=1}^N M \int_{K_{d+1}} \sum_{\boldsymbol{v} \in g_t k \La \backslash \{\boldsymbol{0}\}} \chi_{ \widetilde{B}_m}  ~\wrt k,
 \end{align} which is an inequality involving only characteristic functions on case-one, case-two, case-three, and case-four rods, and thus our results in the beginning of this section (Section~\ref{secSiegelEquidisProof}) for these rods apply as follows.  

\medskip

\noindent{\bf The main term.}  We now estimate the first term of the right-hand side of (\ref{eqnSplitMainandNull}), which we refer to as the \textit{main term}.  Consider the \textit{ancillary step functions} \[f_{j,n}' := \sum_{i=1}^n \big(\sup f \big{\vert}_{B_i}-M\big) \chi_{B_i}.\]  These $f_{j,n}' $ are dominated by $2M \chi_{\overline{B}_0}$.  By the disjointness of the $\{B_i\}$, the function \[f_j':= \lim_{n\rightarrow \infty} f_{j,n}'\] is well-defined (and the limit denotes pointwise convergence).  Now define \[F_{j,n}:=\sum_{i=1}^n M \chi_{B_i}\] and \[F_j := \lim_{n\rightarrow\infty} F_{j,n},\] which is also well-defined function.  Moreover, they satisfy \begin{align}\label{eqnAncillaryLowBnd}
f_j' =  \lim_{n\rightarrow \infty} \widetilde{f}_{j,n}  - F_j.
\end{align}  We may apply the proof that we used to derive (\ref{eqnLimitStepFct}) with $f_j'$ in place of $f_j$ and $f_{j,n}'$ in place of $f_{j,n}$ to deduce \begin{align*}
 \int_{K_{d+1}} \widehat{f_{j \ }'}(g_t k \La) ~\wrt k =  \lim_{n \rightarrow \infty} \int_{K_{d+1}} \sum_{\boldsymbol{v} \in g_t k \La \backslash \{\boldsymbol{0}\}} f_{j,n}' ~\wrt k = \lim_{n \rightarrow \infty}  \int_{K_{d+1}} \widehat{f_{j,n}'}(g_t k \La) ~\wrt k.
\end{align*}  Moreover, these ancillary step functions satisfy \[f_{j,n}' \geq f_{j,n+1}' \geq \cdots \geq -2 M \chi_{\overline{B}_0}\] for all $n$ and, consequently, we have \begin{align*}
 \int_{K_{d+1}} \sum_{\boldsymbol{v} \in g_t k \La \backslash \{\boldsymbol{0}\}} f_{j,n}' ~\wrt k \geq \int_{K_{d+1}} \sum_{\boldsymbol{v} \in g_t k \La \backslash \{\boldsymbol{0}\}} f_{j,n+1}'\geq \textrm{const},
\end{align*} which forms a monotonically decreasing sequence of real numbers (for every fixed $t$) and thus converges as $n \rightarrow \infty$ to a limit that is less than any element in the sequence.  Namely, we have that \begin{equation*}
\lim_{t \rightarrow \infty} \lim_{n \rightarrow \infty} \int_{K_{d+1}} \sum_{\boldsymbol{v} \in g_t k \La \backslash \{\boldsymbol{0}\}} f_{j,n}' ~\wrt k \leq  \lim_{t \rightarrow \infty} \int_{K_{d+1}} \sum_{\boldsymbol{v} \in g_t k \La \backslash \{\boldsymbol{0}\}} f_{j,n}' ~\wrt k 
\end{equation*} for every $n$.  Applying our results for characteristic functions on case-one, case-two, case-three, and case-four balls, we have\begin{align*}
\lim_{t \rightarrow \infty}  \int_{K_{d+1}} \widehat{f_{j \ }'}(g_t k \La) ~\wrt k \leq \sum_{i=1}^n \big(\sup f \big{\vert}_{B_i}-M\big) \vol(B_i)
\end{align*} for every $n$.  Letting $n \rightarrow \infty$, we have \begin{align}\label{eqnAncillaryfct}
\lim_{t \rightarrow \infty}  \int_{K_{d+1}} \widehat{f_{j \ }'}(g_t k \La) ~\wrt k \leq \int_{{\bb R}^{d+1}} f_j ~\wrt {\bf v} -\sum_{i=1}^\infty M \vol(B_i)
\end{align}  Here we have used the properties of the Lebesgue integral, the fact that the $\{B_i\}$ are pairwise disjoint, and the fact that $\D$ has zero Lebesgue measure.
 
An easy modification of the above argument for $ F_{j,n}$ in place of $f_{j,n}'$ yields  \begin{align}\label{eqnLowBnd}
\lim_{t \rightarrow \infty}  \int_{K_{d+1}} \widehat{F_j}(g_t k \La) ~\wrt k \leq  \sum_{i=1}^\infty M \vol(B_i).
\end{align}   
Finally, applying (\ref{eqnAncillaryLowBnd}), (\ref{eqnAncillaryfct}), (\ref{eqnLowBnd}), and the fact that the argument used to deduce (\ref{eqnLimitStepFct}) also works for $\widetilde{f}_{j,n}$  implies the following:\begin{align} \label{eqnMainEst}
\lim_{t \rightarrow \infty}\lim_{n \rightarrow \infty} \sum_{i=1}^n \big(\sup f \big{\vert}_{B_i}\big) \int_{K_{d+1}} \sum_{\boldsymbol{v} \in g_t k \La \backslash \{\boldsymbol{0}\}} \chi_{B_i} ~\wrt k \leq  \int_{{\bb R}^{d+1}} f_j ~\wrt {\bf v}.
\end{align}  This handles the main term.
 
\medskip

\noindent{\bf The null term.}  The second term of the right-hand side of (\ref{eqnSplitMainandNull}), we refer to as the \textit{null term} and handle exactly as $F_{j,n}$ to obtain  \begin{align*}
\lim_{t \rightarrow \infty}\lim_{N \rightarrow \infty} \sum_{m=1}^N M \int_{K_{d+1}} \sum_{\boldsymbol{v} \in g_t k \La \backslash \{\boldsymbol{0}\}} \chi_{ \widetilde{B}_m}  ~\wrt k \leq  \sum_{i=m}^\infty M \vol(\widetilde{B}_m),
\end{align*} where we note that it is immaterial that the $\{\widetilde{B}_m\}$ may not be pairwise disjoint by the second part of Theorem~\ref{theoVitaliCoveringLemma}.  Finally, letting $\gamma \rightarrow 0$ in the proceeding (which bounds the null term from above by $0$) and applying (\ref{eqnSplitMainandNull}) and (\ref{eqnMainEst}) yields our desired result:

 \begin{align*}
\lim_{t \rightarrow \infty} \int_{K_{d+1}} \widehat{f_j}(g_t k \La) ~\wrt k \leq  \int_{{\bb R}^{d+1}} f_j ~\wrt {\bf v}.
 \end{align*}  
 
 \begin{rema}
Instead of introducing ancillary step functions, one could use the compactness of $\overline{B}_0$ to obtain a finite subcover of rods in $\Pa_j$ and $\{\widetilde{B}_m\}_{m=1}^\infty$ to approximate $f_j$.
\end{rema}

\subsubsection{Proof for the general case.}  The $f_j$'s are dominated by $M \chi_{\overline{B}_0}$.  Using an analogous argument to that for $f_{j,n}'$, we have that 
\begin{equation*}
\lim_{t \rightarrow \infty} \lim_{j \rightarrow \infty} \int_{K_{d+1}} \sum_{\boldsymbol{v} \in g_t k \La \backslash \{\boldsymbol{0}\}} f_{j} ~\wrt k \leq  \lim_{t \rightarrow \infty} \int_{K_{d+1}} \sum_{\boldsymbol{v} \in g_t k \La \backslash \{\boldsymbol{0}\}} f_{j} ~\wrt k 
\end{equation*} for every $j$.  Now, applying our result for step functions, we have \begin{align*}
\lim_{t \rightarrow \infty} \int_{K_{d+1}} \widehat{f}(g_t k \La) ~\wrt k \leq  \int_{{\bb R}^{d+1}} f_j ~\wrt {\bf v}
\end{align*} for every $j$.  Finally, we apply dominated convergence and Lemma~\ref{lemmApproxStepFct} to obtain our desired result:  \begin{align*}
\lim_{t \rightarrow \infty} \int_{K_{d+1}} \widehat{f}(g_t k \La) ~\wrt k \leq  \lim_{j \rightarrow \infty}  \int_{{\bb R}^{d+1}} f_j ~\wrt {\bf v} = \int_{{\bb R}^{d+1}} f ~\wrt {\bf v}.
\end{align*}

\section{Nonuniform spiraling: proof of Theorem \ref{theomaineg}}\label{sec:nonuniform}

In this section, we prove Theorem \ref{theomaineg} by using continued fractions to construct a family of one-dimensional examples for which the directions are not uniformly distributed.  For higher dimensions, we use non-minimal toral translations as examples.

We can strengthen the conclusion of Theorem~\ref{theomaineg}:

\begin{Theorem}\label{theomainegstronger}
Let $d \geq 1$.  There exists a lattice $\Lambda \in \SL_{d+1}(\bb R) / \SL_{d+1}(\bb Z)$ and a set $A$ of $\bb S^{d-1}$ for which \begin{eqnarray*} \label{eqnmainegbig}
\liminf_{T \rightarrow \infty}  \frac {\#\{\Lambda \cap R_{A, \epsilon, T}\}}{\#\{\Lambda \cap R_{\epsilon, T}\}} > \vol(A)
\end{eqnarray*}
and \begin{eqnarray*} \label{eqnmainegsmall}
\limsup_{T \rightarrow \infty}  \frac {\#\{\Lambda \cap R_{-A, \epsilon, T}\}}{\#\{\Lambda \cap R_{\epsilon, T}\}} < \vol(A)
\end{eqnarray*}
for every $1 > \epsilon\geq0$.

\end{Theorem}

\subsection{Proof of Theorems~\ref{theomaineg} and~\ref{theomainegstronger} in dimension one}\label{secPMEg1D}

We prove Theorem~\ref{theomainegstronger}, which also suffices to show Theorem~\ref{theomaineg}.  We must construct a lattice in $\bb R^2$ and pick a set $A$ of $\bb S^0$ for that lattice.  Let \[A = \{-1\}.\]  To construct the lattice, we construct a number $x \in \bb R \backslash \bb Q$ using continued fractions (see~\cite{Kh} for an introduction) and form the associated unimodular lattice $\Lambda_x$.  At the end, we will note that our method of construction provides a family of numbers, corresponding to a family of lattices, which satisfy the theorem.  Using an analogous construction allows us to consider $A = \{1\}$ too.

Let $x$ be the irrational number between zero and one for which 

\begin{equation*}
	a_n := \begin{cases}
	4 & \text{if } n \text{ is odd,} \\ 
	n^n & \text{if } n \text{ is even}
\end{cases} 
\end{equation*}
is the $n$th continued fraction element (note $n \geq 1$).  Since $x$ is irrational, there is an unique $p \in \bb Z$ for which $|qx -p| < 1/2$, which, by forgetting $p$, we can regard as a rotation of the circle $\bb R / \bb Z$ by the unique representative of $qx$ in the interval $(-1/2, 1/2)$.  And therefore the only lattice points that matter for the region $R$ from (\ref{cone}) are those $(p, q) \in \bb Z \times \bb Z$ coming from this rotation.  Also, since the negation of a lattice point in $R$ stays in $R$, we may, without loss of generality, consider lattice points with $q \in \bb N$.  Finally, since $x$ is positive, our lattice points will have $p \in \bb N \cup \{0\}$.


For some pairs of such $(p,q)$, the ratio $\frac{qx-p} {\|qx-p\|}$ will be $1$ and for others $-1$, which is equivalent to asking whether $(p,q)$ is on one or the other side of the ray starting at the origin and going through $(x, 1)^T$, which is equivalent to asking whether $qx -p >0$ or $qx -p <0$, and which, if, for conciseness, we introduce the notation \[q \cdot x\] to denote the circle rotation above, is equivalent to asking whether $q \cdot x >0$ or $q \cdot x <0$.

Let $p_n/q_n$ denote the $n$th convergent of $x$.  We will use the following well-known facts about continued fractions and circle rotations:

\begin{enumerate}
	\item \label{cfFact1} The rotations $q_{n-1} \cdot x$ and $q_n \cdot x$ alternate in sign.
	\item \label{cfFact2}\begin{align*} p_n &= a_n p_{n-1} + p_{n-2} \\ q_n &= a_n q_{n-1} + q_{n-2}
\end{align*}
\item \[q_n p_{n-1} - p_n q_{n-1} = -1^n\]
\item\label{cfFact4} \[\frac 1 {q_n + q_{n+1}} < |q_n \cdot x| < \frac 1 {q_{n+1}}\]
\item Convergents are best approximates (of the second kind):  \[|q_n \cdot x| < |q \cdot x|\] for all $0< q < q_{n+1}$.
\end{enumerate}

The following is a general fact of the continued fraction of any irrational number:

\begin{lemm} We have \[\frac {a_{n+1}} 2 < \frac{|q_{n-1} \cdot x|} {|q_n \cdot x|}< a_{n+1} +2.\]
\end{lemm}

\begin{proof}
Both inequalities follow from Facts~(\ref{cfFact2}) and (\ref{cfFact4}).
\end{proof}

For our particular number $x$, the lemma implies that 

\begin{coro}\label{coroRatioGaps}
For $n$, an even number, we have \[2 < \frac{|q_{n-1} \cdot x|} {|q_n \cdot x|}< 6\] and, for $n$, an odd number, we have \[\frac {(n+1)^{n+1}} 2 < \frac{|q_{n-1} \cdot x|} {|q_n \cdot x|}< (n+1)^{n+1} +2.\]
\end{coro}

Since $x < 1/2$, we have that $1 \cdot x > 0$.  Using facts about continued fractions, the usual conventions $q_{-1}=0$, $p_{-1}=1$, and that $p_0=0$ by construction, it follows that $q_0 =1$.  Consequently,

\begin{lemm}\label{lemmDirections} For $n$, an even integer, we have \[q_n \cdot x >0 \] and, for $n$, an odd integer, we have \[q_n \cdot x <0.\]
\end{lemm}
\begin{proof}
As noted, $q_0 \cdot x >0$.  Fact (\ref{cfFact1}) immediately implies the result.
\end{proof}

Since the denominators of the convergents are strictly increasing, Fact (\ref{cfFact4}) implies that $(p_n, q_n)$ are in $R$ and that, for $n$ large enough, $q_n \cdot x$ is itself a rotation by a small angle (much smaller than angle $x$).  However, there are other lattice points in $R$, which we now describe.

We will count relevant lattice points by induction; it is convenient to induct on $n$, the index of the convergents.  Since we are considering a limit, we may start counting lattice points starting with some large $q_n$ without affecting our result.  Therefore, we may assume that $\frac {1} {q_n}$ is small.  

We are interested in lattice points in $R$.  Recall that these lattice points come from the above-mentioned rotations and hence lattice points in $R$ are equivalent to rotations $q\cdot x$ for which $|q\cdot x |\leq \frac 1 {q}$.  To help us count, let us enlarge the lattices points of interest to those corresponding to \begin{align}\label{eqnBiggerCount} |q \cdot x| \leq \frac 1 {q_{n}} \end{align} for $q_{n} \leq q < q_{n+1}$ and exclude those not in $R$. It follows from Fact (\ref{cfFact4}) that the only lattice points satisfying (\ref{eqnBiggerCount}) from those corresponding to $0 \leq q < q_n$ are $q_{n-1}$ and $2 q_{n-1}$ on one side of $0$, $0$ itself, and one on the other side of $0$, which we will say corresponds to $\tilde{q}$---since convergents are best approximates, we know that $|\tilde{q} \cdot x| > |q_{n-1} \cdot x|$.  

For the initial step of the induction on $n$, we have chosen to ignore the lattice points corresponding to $q_{n-1}$, $2 q_{n-1}$, $0$, and $\tilde{q}$ and, for an induction step, we have already counted the contribution from these points.  It the in-between lattice points corresponding to $q_{n} \leq q < q_{n+1}$ that concern us.  The division algorithm describes all such lattice points  as follows:  $q = m q_n + r$.  The only remainders $r$ of interest are the ones already chosen, namely $q_{n-1}$, $2 q_{n-1}$, $0$, and $\tilde{q}$.  By Fact (\ref{cfFact2}) applied to $x$, there are always at least three in-between points---to be precise, these in-between points for a given remainder $r$ correspond to $\{q_n + r, 2 q_n + r, 3 q_n +r, \cdots\}$.  Moreover, the number of in-between points is either $a_{n+1}-1$ or $a_{n+1}$ depending on the remainder $r$.  If 
$r_1 \cdot x$ and $r_2 \cdot x$ are adjacent on the circle for $0 \leq r_1 \neq r_2 <q_n$, then their in-between lattice points divides the interval formed by $r_1 \cdot x$ and $r_2 \cdot x$ up into equal length pieces with the sole exception of one piece which may be slightly longer---this observation follows from Fact (\ref{cfFact1}) and the fact that convergents are best approximates and rotations are isometries.  

Let us consider these in-between points.  We claim that the only in-between point for the remainder $r = \tilde{q}$ that may be relevant corresponds to $q_n + \tilde{q}$.  First note, by Fact (\ref{cfFact1}), the fact that circle rotation is translation on the abelian group $\bb R / \bb Z$, and that this translation is an isometry, we have that $|(m q_n + \tilde{q}) \cdot x | > |(\ell q_n + \tilde{q}) \cdot x|$ for $a_{n+1} > m > \ell\geq 0$.  For $m \geq 2$, we have \[|(m q_n + \tilde{q}) \cdot x | > \frac 1 {2 q_n}\] by Fact (\ref{cfFact4}), but we also have \[\frac 1 {m q_n + \tilde{q}} < \frac 1 {2 q_n},\] which shows our claim.   The lattice point corresponding to $q_n + \tilde{q}$ is only one point and may be ignored for the limit that we are computing.

We claim that the only in-between point for remainder $r = 2 q_{n-1}$ that may be relevant corresponds to $q_n + 2 q_{n-1}$.  The proof is analogous to that for $
\tilde{q}$.  And the possible relevant lattice point can be ignored for the limit.

For the remaining two remainders, the behaviors differ greatly (by construction) for odd-indexed and even-indexed convergents; we consider these cases separately.

\subsubsection{Odd-indexed convergents.}\label{susubsecOIC}  Let $n$ be odd. We will show that many of the in-between lattice points for the remainder $0$ are in $R$, while very few of the in-between lattice points for remainder $q_{n-1}$ are.  Let us first consider the in-between points for $q_{n-1}$, which correspond to $\{m q_n + q_{n-1}\}$ for $0 < m < a_{n+1}$ by Fact (\ref{cfFact2}).  By Lemma~\ref{lemmDirections}, we have \[0 < (m q_n + q_{n-1}) \cdot x < (\ell q_n + q_{n-1}) \cdot x < q_{n-1} \cdot x\] for $0 <  \ell < m < a_{n+1}$.  Note that we have exactly $M := a_{n+1} -1$ in-between points between $q_{n-1} \cdot x$ and $0$ in the given range.  Since circle rotation by $x$ is an abelian group, these in-between points divide up the interval between $0$ and $q_{n-1}\cdot x$ into equal length segments, except for the segment with $0$ as an endpoint, which is slightly longer.  Thus, we have  \begin{equation}\label{eqnAlmostEvenGaps}
q_{n-1} \cdot x - \frac m {M+1} (q_{n-1} \cdot x) < (m q_n + q_{n-1}) \cdot x
\end{equation}
 for $0 < m \leq M$.  Now to be excluded from $R$, a lattice point must satisfy the following condition \[\frac 1 {m q_n + q_{n-1}} <  (m q_n + q_{n-1}) \cdot x, \] which is satisfied, as one can see by applying Fact (\ref{cfFact4}) to (\ref{eqnAlmostEvenGaps}), if the point satisfies \[\frac 1 {m q_n} < \frac 1 {2 q_n}  \bigg(1 - \frac m {M+1}\bigg).\]  The latter condition, in turn, is equivalent to asking at which values of $m$ is the parabola $-m^2 + (M+1)m - 2(M+1) >0$.  The answer is between the two roots, which, for $M$ large enough, are as close as we like to $2$ and $M-1$.  Since $n$ is odd, $M$ can be chosen large.  Thus, except, possibly, for four in-between points, the rest are excluded from $R$.  We can ignore these four points for computing the limit.
 
Finally, to finish the odd-indexed case, we consider in-between points for $0$. There are $N := a_{n+1}$ of such points in the given range (which divide up the segment between $\tilde{q}\cdot x$ and $0$ into equal length pieces, except for a slightly longer piece with endpoint $\tilde{q}\cdot x$).  From~ (\ref{eqnAlmostEvenGaps}), we have that 
\begin{eqnarray}\label{eqnAlmostEvenGaps2}
	- \frac m {N} (q_{n-1} \cdot x) < (m q_n) \cdot x 
\end{eqnarray}
for $0< m < N$. Now for a lattice point to be in $R$, we need the following condition to hold:  \[-\frac 1 {m q_n} \leq (m q_n) \cdot x,\] which is satisfied, as one can see by applying Fact (\ref{cfFact4}) to (\ref{eqnAlmostEvenGaps2}), if the point satisfies \[-\frac 1 {m q_n} \leq  - \frac m {q_nN}.\]  Let $L_n$ be the number of in-between lattice points for remainder $0$ in $R$.  Our calculation implies that $L_n \geq \lfloor(n+1)^{(n+1)/2} \rfloor$.  

\subsubsection{Even-indexed convergents.}  For any remainder, there are at most $a_{n+1} =4$ in-between points.  All of these can be ignored in the limit calculation.

\subsubsection{Finishing the proof of Theorem~\ref{theomainegstronger}}  Thus, the lattice points that project onto $A$ have count $\sum (L_{2k+1} + P_1)$ where $P_1$ is a natural number $\leq5$.  While, all lattice points in $R$ have count $\sum (L_{2k+1} + P_2)$ where $P_2$ is a natural number $\leq 10$. It is clear that fixing an $\epsilon$ does not affect the proceeding.  Therefore, we have shown that \[\liminf_{T \rightarrow \infty}  \frac {\#\{\Lambda \cap R_{A, \epsilon, T}\}}{\#\{\Lambda \cap R_{\epsilon, T}\}} = 1\] and 
\[\limsup_{T \rightarrow \infty}  \frac {\#\{\Lambda \cap R_{-A, \epsilon, T}\}}{\#\{\Lambda \cap R_{\epsilon, T}\}} =0.\] The theorem is now immediate.

\subsection{Other numbers that satisfy Theorem~\ref{theomainegstronger} in dimension one}  It is clear that our construction of $x$ via continued fractions in the proof of Theorem~\ref{theomainegstronger} is a general construction.  Let $a_n(x_1)$ and $a_n(x_2)$ be the $n$-th elements of the continued fraction expansions of $x_1$ and $x_2$, respectively.  Then we define $x_1 \# x_2$ to be the continued fraction whose elements are \begin{align*} a_{2n}(x_1 \#x_2) :=& a_n(x_1) \\ a_{2n+1}(x_1\#x_2):=& a_n(x_2) \end{align*} for $n \in \bb N \cup\{0\}.$  Since continued fractions are unique, the operation $\#$ is a well-defined (noncommutative) product of real numbers, which we refer to as the {\it continued fraction product}.   It is clear from our construction above that, to satisfy the theorem, the number $x_1 \# x_2$ must have $a_n(x_2)$ grow faster than $a_n(x_1)$ as $n \rightarrow \infty$. Let us refer to such numbers as {\it unbalanced}.

Finally, reversing the order of the continued fraction product for our constructed number $x$ will provide an example of a number satisfying the theorem for $A = \{1\}$.

\subsection{Proof of Theorems~\ref{theomaineg} and~\ref{theomainegstronger} in higher dimensions}   We use the well-known fact:  \begin{lemm} The toral translation by a vector $\boldsymbol{x}=(x_1, \cdots, x_d)^T$ is non-minimal if and only if there exist integers $k_1, \cdots, k_d$ not all zero such that $\sum k_i x_i \in \bb Z$. 
\end{lemm} 

This is example is a simple observation. Let $d \geq 1$.  Let $\boldsymbol{x}$ correspond to a non-minimal toral translation.  Then, it follows that there is a primitive integer lattice vector $\boldsymbol{v}$ perpendicular (with respect to the usual dot product in $\bb R^{d+1}$) to the $d+1$-vector $(\boldsymbol{x}^T, 1)^T$.  And $\boldsymbol{v} \neq (0, \cdots, 0, 1)^T$.  Changing the basis of $\bb Z^{d+1}$ to $\{\boldsymbol{v}, \boldsymbol{v}_2, \cdots, \boldsymbol{v}_{d+1}\}$ allows us to see that $(\boldsymbol{x}^T, 1)^T$ lies in the $X:= \textrm{span}\{\boldsymbol{v}_2, \cdots, \boldsymbol{v}_{d+1}\}$ (thought of as a subspace of $\bb R^{d+1}$).  Now it follows that the basis vectors $\{\boldsymbol{v}, \boldsymbol{v}_2, \cdots, \boldsymbol{v}_{d+1}\}$ determine a parallelepiped of $d+1$-volume equal to $1$.  Hence it follows that the only lattice points of $\bb Z^{d+1}$ closer to $(\boldsymbol{x}^T, 1)^T$ than the Euclidean distance between $v$ and $X$ must lie on $X$.

Now let $Y:=\textrm{span}\{\boldsymbol{e}_1, \cdots \boldsymbol{e}_d\}$ (thought of as a subspace of $\bb R^{d+1}$).  The spaces $X$ and $Y$ do not coincide because their normal vectors are not in the same direction.  Therefore $X \cap Y \cap \bb Z^{d+1}$ is a proper sublattice  of $X \cap \bb Z^{d+1}$ and hence no lattice points in the thinning region $R$ (after becoming thin enough) project onto $\bb S^{d-1}$ outside of this sublattice---the projection is onto a lower dimensional sphere $\bb S^{d-2}$.  This proves the theorem for $d \geq 2$.

For $d =1$, we note that the $(x, 1)^T$ is a rational vector and hence goes through an point of $\bb Z^2$.  It easy to see that all integer lattice points close enough to $(x,1)^T$ lie on the line through it---in this case, there is no projection at all.  This proves the theorem for $d=1$.  

We conclude by remarking that an exaggerated version of the proceeding example is given by taking the unbalanced number $x$ constructed in the proof of Theorem~\ref{theomainegstronger} in dimension one and forming $\boldsymbol{x} = (x, \cdots, x)^T$.  This gives a higher dimensional example satisfying Theorem~\ref{theomainegstronger}.


\end{document}